\documentclass[12pt,oneside]{amsart}
\usepackage{amssymb, amsmath}
\usepackage{float}
\usepackage{pdfpages}
\usepackage{epstopdf}
\usepackage{comment}
\usepackage{booktabs}
 \usepackage[abs]{overpic}
\usepackage{graphicx}
\usepackage{color}
\usepackage{pinlabel}
\usepackage{mathtools,xparse}
\usepackage{tikz}
\usetikzlibrary{arrows,positioning,shapes,fit,calc}
\usepackage{enumerate}

\usepackage[colorlinks = true,
            linkcolor = red,
            urlcolor  = red,
            citecolor = red,
            anchorcolor = red]{hyperref}
\usepackage{amsrefs}
\usepackage[pagewise]{lineno}

\usepackage{pinlabel}

\theoremstyle{plain}
\newtheorem{theorem}{Theorem}[section]

\newtheorem*{theorem*}{Theorem}
\newtheorem*{maintheorem-intro}{Theorem}
\newtheorem*{maintheorem-intro-2}{Theorem~\ref{Bridge number and genus}}
\newtheorem*{theorem-cablingconj}{Theorem~\ref{apps1} (1)}
\newtheorem*{theorem-toroidal}{Specialization of Theorem~\ref{apps1} (2)}
\newtheorem*{theorem-lens}{Theorem~\ref{Bounding distance - special}(1)}
\newtheorem*{theorem-SFS}{Theorem~\ref{Bounding distance - special}(2)}
\newtheorem*{theorem-cosmetic}{Theorem}
\newtheorem*{theorem-bridge}{Specialization of Corollary~\ref{Cor: exceptional bridge}}
\newtheorem*{theorem-Heeggenus}{Corollary~\ref{Cor: exceptional bridge} (2)}

\newtheorem{proposition}[theorem]{Proposition}
\newtheorem{corollary}[theorem]{Corollary} 
\newtheorem{lemma}[theorem]{Lemma}

\theoremstyle{definition}

\newtheorem{definition}[theorem]{Definition}

\newtheorem{conjecture}{Conjecture}
\newtheorem*{ackn}{Acknowledgements}
\theoremstyle{definition}

\newcommand{\R}{{\mathbb R}}

\newcommand{\Z}{\mathbb Z}

\newcommand{\lat}{\mathbb{L}^3}

\newcommand{\bi}{\begin{itemize}}
\newcommand{\ei}{\end{itemize}}
\newcommand{\be}{\begin{enumerate}}
\newcommand{\ee}{\end{enumerate}}

\newcommand{\defn}[1]{\emph{#1}}

\begin{document}
   % title

   \title[Vertex Distortion of Lattice Knots]{Vertex Distortion of Lattice Knots}
  
   \author{Marion Campisi}
 \email{marion.campisi@sjsu.edu}
 \author{Nicholas Cazet}
 \email{nccazet@ucdavis.edu}

   % Note that the short title for running heads goes in square
   % brackets.  This is optional.  The long title goes in curly
   % braces.  In the long title, line breaks are indicated by \\

\begin{abstract}
The vertex distortion of a lattice knot is the supremum of the ratio of the distance between a pair of vertices along the knot and their distance in the $\ell_1$-norm.  Inspired by Gromov, Pardon and Blair-Campisi-Taylor-Tomova, we show that results about the distortion of smooth knots hold for vertex distortion:  the vertex distortion of a lattice knot is 1 only if it is the unknot, and there are minimal lattice-stick number knot conformations with arbitrarily high distortion.
\end{abstract}
\maketitle
\date{\today}

\section{Introduction }

A \defn{polygonal knot} is a knot that consists of line segments called \defn{sticks}.  A \defn{lattice knot} is a polygonal knot in the cubic lattice $\mathbb{L}^3=(\R\times\Z\times\Z)\cup(\Z\times\R\times\Z)\cup (\Z\times\Z\times\R).$  All knots and curves in this paper are taken to be tame.

 The \defn{vertex set} of a lattice knot $K$, denoted $V(K)$, is the set of points ${\bf p} \in K\cap (\Z \times \Z \times \Z)$. The vertex set of a tame lattice knot is finite. Points in $\mathbb{L}^3$ will be denoted by ${\bf p}=(p_1,p_2,p_3)$.

For smoothly embedded knots one can assign a value called the {\it distortion}: \[ \delta(K)=\sup\limits_{a,b\in K} \frac{d_K(a,b)}{d(a,b)}.\]

Here, we define the \defn{vertex distortion} of a lattice knot $K$ in $\lat$ as

\[
\delta_V(K)=\sup\limits_{{\bf p,q} \in V(K)} \frac{d_{K}({\bf p,q} )}{d_1({\bf p,q})}=\max\limits_{{\bf p,q}\in V(K)} \frac{d_{K}({\bf p,q})}{d_1({\bf p,q})}
\]
 where $d_{K}({\bf p,q})$ denotes the shorter of the two injective paths from ${\bf p}$ to ${\bf q}$ along $K$ and $d_1({\bf p,q})=\sum_{i=1}^3 |p_i-q_i|$, the $\ell_1-$metric. Unlike the distortion of general rectifiable curves, as introduced by Gromov \cite{gromov1983}, the supremum may be replaced with the maximum in our vertex distortion of lattice knots: the set $V(K)$ is finite. We can turn this into a knot invariant by defining for each lattice knot type $[K]$ 
\[
\delta_V([K])=\inf\limits_{K \in [K]} \delta_V(K)
\]
where the infimum is taken over all lattice conformations  $K$  representing the knot type $[K]$.

 The distortion of smooth knots has proven to be a challenging quantity to analyze.  Gromov  showed that $\delta(K)\geq \frac{\pi}{2}$ with equality if and only if $K$ is the standard round circle \cite{gromov1983}. Moreover, Denne and Sullivan showed that $\delta(K)\geq \frac{5\pi}{3}$ whenever $K$ is not the unknot \cite{denne2004distortion}.

In 1983, Gromov  \cite{gromov1983} asked if there is a universal upper bound on $\delta(K)$ for all knots $K$. Pardon \cite{Pardon_2011} answered this question negatively when he showed that the distortion of a knot type is bounded below by a quantity proportional to a certain topological invariant, called representativity.  Blair, Campisi, Taylor, and Tomova showed that distortion is bounded below by bridge number and bridge distance, and exhibit an infinite family of knots for which their bound is arbitrarily stronger than Pardon's.

We show that analogous to \cite{gromov1983}, 

\begin{theorem}
\label{thm:Unknot}
For any knot $K$ in the cubic lattice, if $\delta_V(K)=1$, then $K$ is the unknot.  
%\nicholas{This isn't true, take a rectangle with width two and height one, then the distortion is 3 but it's the unknot. ``ONLY IF" is the only direction that works"}
\end{theorem}

and like \cite{Pardon_2011} and  \cite{Blair_2020},

\begin{theorem}
\label{thm:Unbounded}
There exists a sequence of minimal lattice-stick number torus knots, $\mathcal{K}_{2k}$,  with $\delta_V(\mathcal{K}_{2k})\rightarrow \infty$ as $k\rightarrow \infty$.
\end{theorem}

This paper is structured as follows: Section \ref{section:defns} contains relevant definitions and background. Theorem \ref{thm:Unknot} is proved in Section \ref{section:Unknot}. Section \ref{section:Lattice Torus Knots} tabulates a family of knots used to prove Theorem \ref{thm:Unbounded}. Section \ref{section:torusverification} establishes properties of these knots. In Section \ref{section:main}, the knots are used to prove Theorem \ref{thm:Unbounded}.

\begin{ackn} The authors would like to thank the referee for their careful reading and helpful suggestions.  
\end{ackn}

\section{Definitions}
\label{section:defns}

For the duration of the paper, all knots $K$  are taken to be orientated lattice knots, and $[K]$ is the class of lattice knots isotopic to $K$.

\begin{definition}
A {\it vertex} of $K$ is any point in the vertex set $V(K)=K\cap\mathbb{Z}^3$.
\end{definition}

%Let $r:[0,1]\to\lat$ be a parameterization of a lattice knot. Then $r'(t_0)$ does not exist for finitely many $t_0\in[0,1]$. The tangent vector is undefined at 

\begin{definition} A \defn{stick} of a lattice knot $K$ is a maximal line segment of $K$ in $\lat$, i.e. a line segment contained in $K$ that is not contained in any longer line segment of $K$. \end{definition}

\begin{definition}
For a lattice knot $K$, the endpoints of each stick are vertices of $K$; such a vertex is called a \defn{critical vertex} of the lattice knot. \end{definition}

To clarify, a vertex of lattice knot $K$ is any point in $K \cap \mathbb{Z}^3$ while a critical vertex of $K$ is a point where two consecutive sticks intersect. Consecutive sticks meet at a corner of the polygon that represents $K$, forming a ninety-degree angle. Every critical vertex of $K$ is a vertex of $K$.

Each stick of length $n$ can be decomposed into $n$ unit length line segments.

\begin{definition} An \defn{edge} of a lattice knot is a unit length line segment, subset of a stick, connecting consecutive vertices. \end{definition}

\noindent Consecutive edges are parallel, when they belong to the same stick, or perpendicular, when their intersection is a critical vertex.

\begin{definition} The \defn{edge length} of a lattice knot $K$, denoted $e_L(K)$, is the total number of edges in $K$. \end{definition}

 Endow a lattice knot with an orientation. Then, there is a bijection between the set of edges and the set of vertices by sending each edge to its terminal point. Thus, $|V(K)|=e_L(K)$.  
 
 \begin{definition} The lattice stick number $s_{CL}(K)$ of a lattice stick conformation $K$ is the number of constituent sticks of K. We minimize $s_{CL}(K)$ over all lattice stick conformations $K\in [K]$ and we call this number the lattice stick index $s_{CL}([K])$ of the knot type $[K]$. \end{definition}

Since our lattice knots are taken to be orientated, each stick has a well-defined initial and terminal critical vertex, ${\bf v_0}$ and ${\bf v_t}$ respectively. The point ${\bf v_t}-{\bf v_0}$ is a coordinate triple with zeros in all but one coordinate; the nonzero coordinate's value defines the {\it stick type}: $x^+$, $y^+$, $z^+$, $x^{-}$, $y^{-}$, or $z^{-}$.   

\begin{definition} Let $K$ be an orientated lattice knot. Let ${\bf v_0}$ and ${\bf v_t}$ be the initial and terminal critical vertices of a stick $S\subset K$. We have that ${\bf v_t}-{\bf v_0}=(a,b,c)\in \Z^3$. Then $S$ is an {\it $x^+$-stick} if $a>0$ and a {\it $x^-$-stick} if  $a<0$. A {\it $y^+$-, $y^-$-, $z^+$-,} and {\it $z^-$-stick} is defined analogously where $b$ or $c$ is positive or negative. 
\end{definition}

Heuristically, a stick is parallel to the $x$-, $y$-, or $z$-axis with an orientation dictating whether it has an increasing or decreasing, $x$-, $y$-, or $z$-coordinate, respectively, while all other coordinates are held constant. The coordinate that is changing, along with whether it is increasing or decreasing, gives the stick type.

  Every lattice knot can be described as a sequence of stick lengths and a paired sequence of stick types. Beginning at a critical vertex, the orientation induces an ordering of the sticks types. This orientation also induces a sequence of stick lengths. Every lattice knot can be isometrically translated such that any critical vertex translates to the origin, therefore our convention will be for this construction to begin at the origin.  

\begin{figure}[h]

  \begin{tabular}{ c | *{3}{c} c}

     & $x$ & $y$ & $z$ \\   \cline{1-4}
    1 & 2 & 1 & 3\\ 
    2 & 3 & 2 & 2 &          \hspace{1.5cm} $z^+$, $x^+$, $y^+$, $z^{-}$, $x^{-}$, $y^{-}$, $z^+$, $x^+$, $y^+$, $z^{-}$, $x^{-}$, $y^{-}$\\
      3 & 2 & 3 & 1 \\
         4 & 1 & 2 & 2 \\

\end{tabular}

\caption{Stick length sequence with pairing stick type sequence of the trefoil.}

\label{fig:trefoiltab}

\end{figure} 

The tabulation of stick lengths paired with the stick type sequence in Figure \ref{fig:trefoiltab} gives a construction of the trefoil. The $j^\text{th}$ row of the $x$-column reads the length of the $j^\text{th}$  $x$-stick, similarly for the $y$- and $z$-columns. Beginning at the origin, we see the stick type sequence is initiated by $z^+$-, $x^+$-, and $y^+$-sticks. Since these are the first $x$-, $y$-, and $z$-sticks, we read the lengths of these sticks from the first row, 3, 2, 1, respectively. The following stick type is $z^{-}$. Since this is the second $z$-stick traversed, we read its length from the second row of the table. 

Figure \ref{fig:construct} illustrates the prescribed construction of  \ref{fig:trefoiltab} where critical vertices are shown with squares and generic vertices are shown with dots.

%A square is placed at the critical vertex lying at the origin, the initial point of this construction, and an orientation arrow is shown above the origin exhibiting that the first stick in the sequence is a $z^+$-stick. The vertices of the knot are shown with dots.

\begin{figure}

\begin{overpic}[unit=.65mm,scale=.65]{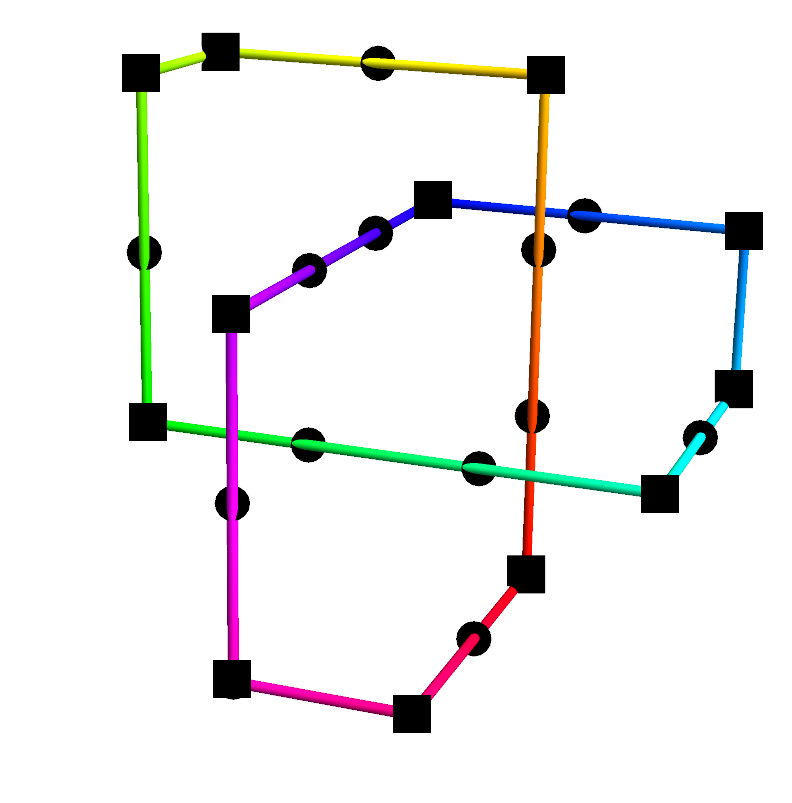}
 \put(-2,30){$s_{CL}(K)=12$}\put(-2,20){$e_{L}(K)=24$}
\end{overpic}
\caption{Trefoil constructed from the tabulation in Figure \ref{fig:trefoiltab}.}
\label{fig:construct}
\end{figure}

%In general, there are as many rows in the tabulation of stick lengths, as seen in Figure \ref{fig:trefoiltab}, as the largest total of either $x$-, $y$-, $z$-sticks. If a knot has more $z$-sticks than $x$- or $y$-sticks, then the $x$ or $y$ columns will be padded with zeros, beginning when the row exceeds the number of $x$- or $y$-sticks. No such padding will be necessary for the knots that we tabulate. 

\begin{definition} A \defn{staircase walk} from ${\bf a}\in \lat$ to ${\bf b}\in \lat$ is a piecewise linear path $${\bf r}:[0,1]\to \lat$$ given by ${\bf r}(t)=(x(t),y(t),z(t))$ such that ${\bf r}(0)={\bf a}$, ${\bf r}(1)={\bf b}$ and each of $x(t)$, $y(t)$, and $z(t)$ is, independently, nondecreasing or nonincreasing. \end{definition}

All possible staircase walks between $(1,1,0)$ and $(3,3,0)$ are shown in Figure \ref{fig:staircase}; since the entry-wise difference of these two points has no change in the $z$-coordinate, all walks are subsets of the $xy$-plane. 

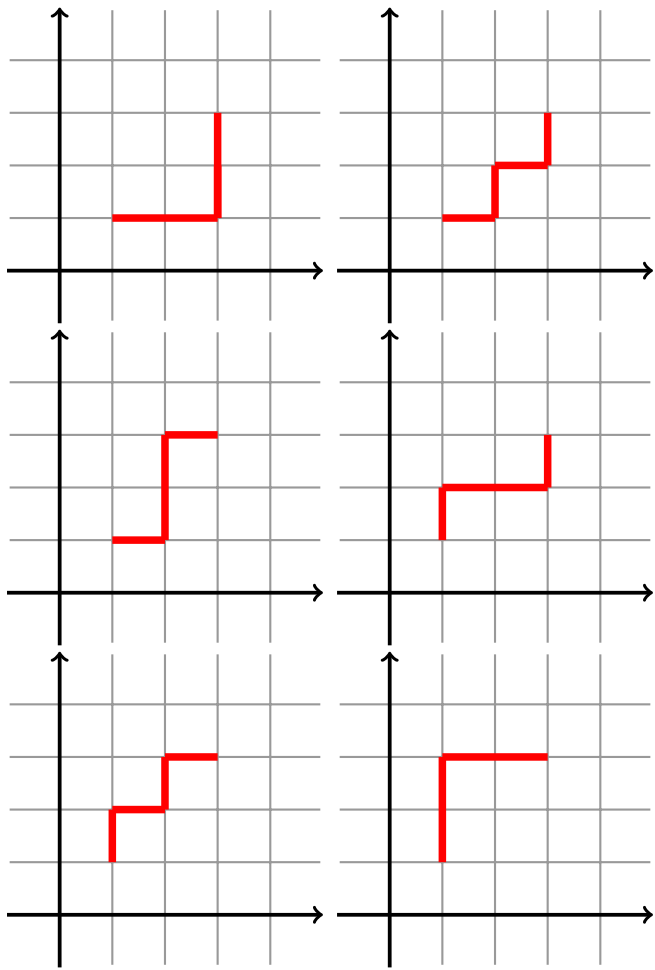
\begin{figure}[]
   \begin{tikzpicture}[scale=.7]

        % help lines
        \draw[step=1,help lines,thick, black!40] (-0.95,-0.95) grid (4.95,4.95);
        % axis
   \draw[line width=0.5mm,->] (-1,0) -- (5,0);
    \draw[line width=0.5mm,->] (0,-1) -- (0,5);
      \draw[line width=1mm,-,red] (1,1)-- (2,1);
        \draw[line width=1mm,-,red] (2,1)-- (3,1);
                \draw[line width=1mm,-,red] (3,1)-- (3,2);
                             \draw[line width=1mm,-,red] (3,1)-- (3,3);
                   
        % points

    \end{tikzpicture}
   \begin{tikzpicture}[scale=.7]

        % help lines
           \draw[step=1,help lines,thick, black!40] (-0.95,-0.95) grid (4.95,4.95);
                   % axis
        \draw[line width=0.5mm,->](-1,0) -- (5,0);
        \draw[line width=0.5mm,->](0,-1) -- (0,5);
         \draw[line width=1mm,-,red] (1,1)-- (2,1);
        \draw[line width=1mm,-,red] (2,1)-- (2,2);
               \draw[line width=1mm,-,red] (2,2)-- (3,2);
                               \draw[line width=1mm,-,red] (3,2)-- (3,3);
                   
        % points

    \end{tikzpicture}
 \begin{tikzpicture}[scale=.7]

        % help lines
           \draw[step=1,help lines,thick, black!40] (-0.95,-0.95) grid (4.95,4.95);
        % axis
        \draw[line width=0.5mm,->] (-1,0) -- (5,0);
        \draw[line width=0.5mm,->] (0,-1) -- (0,5);
        \draw[line width=1mm,-,red] (1,1)-- (2,1);
        \draw[line width=1mm,-,red] (2,1)-- (2,3);
                 \draw[line width=1mm,-,red] (2,3)-- (3,3);

        % points

    \end{tikzpicture}
         \begin{tikzpicture}[scale=.7]

        % help lines
                 \draw[step=1,help lines, thick, black!40] (-0.95,-0.95) grid (4.95,4.95);
        % axis
         \draw[line width=0.5mm,->](-1,0) -- (5,0);
          \draw[line width=0.5mm,->](0,-1) -- (0,5);
         \draw[line width=1mm,-,red] (1,1)-- (1,2);
         \draw[line width=1mm,-,red] (1,2)-- (3,2);
                 \draw[line width=1mm,-,red] (3,2)-- (3,3);

        % points

    \end{tikzpicture}
       \begin{tikzpicture}[scale=.7]

        % help lines
             \draw[step=1,help lines,thick, black!40] (-0.95,-0.95) grid (4.95,4.95);
        % axis
               \draw[line width=0.5mm,->](-1,0) -- (5,0);
        \draw[line width=0.5mm,->](0,-1) -- (0,5);
        \draw[line width=1mm,-,red] (1,1)-- (1,2);
      \draw[line width=1mm,-,red] (1,2)-- (2,2);
               \draw[line width=1mm,-,red] (2,2)-- (2,3);
                
      \draw[line width=1mm,-,red] (2,3)-- (3,3);

        % points

    \end{tikzpicture}
       \begin{tikzpicture}[scale=.7]

        % help lines
            \draw[step=1,help lines,thick, black!40] (-0.95,-0.95) grid (4.95,4.95);
        % axis
         \draw[line width=0.5mm,->](-1,0) -- (5,0);
            \draw[line width=0.5mm,->](0,-1) -- (0,5);
     \draw[line width=1mm,-,red] (1,1)-- (1,3);
         \draw[line width=1mm,-,red] (1,3)-- (3,3);

        % points

    \end{tikzpicture}

\caption{All staircase walks between $(1,1,0)$ and $(3,3,0)$.}
\label{fig:staircase}

\end{figure}
\noindent For ${\bf a}=(a_1,a_2,a_3), {\bf b}=(b_1,b_2,b_3)\in\Z^3$, the number of staircase walks from ${\bf a}$ to ${\bf b}$ is  $$\frac{d_1({\bf a,b})!}{(|a_1-b_1|)!(|a_2-b_2|)!(|a_3-b_3|)!}.$$

 It is important to note that a staircase walk represents a most efficient path in the lattice from one vertex to another, i.e. for ${\bf a,b}\in V(K)$, all staircase walks from ${\bf a}$ to ${\bf b}$ are of length $d_1({\bf a},{\bf b})$ and no lattice path from ${\bf a}$ to ${\bf b}$ has a shorter length.

 \begin{lemma} A path from ${\bf a}\in \Z^3\subset \lat $ to ${\bf b}\in \Z^3\subset \lat$, for ${\bf a} \neq {\bf b}$, has length $d_1({\bf a,b})$ if and only if it is a staircase walk. \end{lemma}
 
 \begin{proof} 
   Through an isometry, we may assume that $a_i\leq b_i$ for all $i\in \{1,2,3\}$. Suppose ${\bf r}$ is a path in $\lat$ from ${\bf a} =(a_1,a_2,a_3)\in \Z^3$ to ${\bf b}=(b_1,b_2,b_3)\in \Z^3$.
  
  The sum of the $x^+$-stick lengths must be at least $b_1-a_1$. Likewise, the sum of the $y^+$- and $z^+$-stick lengths must be at least $b_2-a_2$ and $b_3-a_3$, respectively. The length of the path is the sum of the stick lengths. Let $s$ be the sum of the negative stick lengths. Then the length of the path is greater than or equal to $(b_1-a_1)+(b_2-a_2)+(b_3-a_3)+s=d_1({\bf a,b})+s$. 
  
  If we assume that the length of the path is $d_1({\bf a,b})$, then $s=0$ implying that no negative sticks exist. Therefore, all coordinate functions of $r$ are nondecreasing, i.e. $r$ is a staircase walk.
  
If we assume that the path is a staircase walk, then  all coordinate functions of ${\bf r}$ are nondecreasing, since $a_i\leq b_i$ for all $i\in \{1,2,3\}$. This path will be comprised of only $x^+$-, $y^+$-, and $z^+$-sticks. Thus, the sum of the $x$-stick lengths is $b_1-a_1$. Likewise, the sum of the $y$-, and $z$-stick lengths is $b_2-a_2$ and $b_3-a_3$, respectively. Thus, the length of the path is $d_1({\bf a,b})=(b_1-a_1)+(b_2-a_2)+(b_3-a_3)$.

 \label{lem:staircasedistance}
 
 \end{proof}

\begin{definition}The \defn{minimal bounding box} of $K$ is the box, $[x_1,x_2]\times[y_1,y_2]\times[z_1,z_2]$, of smallest volume that contains $K$. The points $\{(x_i,y_j,z_k): i,j,k\in\{1,2\}\}$ are \defn{corners} of the minimal bounding box.\end{definition} 

The minimal bounding box of a lattice knot will have integer endpoints for all intervals in its product. 

\begin{lemma} A (critical) vertex ${\bf v}=(v_1,v_2,v_3)\in V(K)$ is a corner of the minimal bounding box of a lattice knot $K$ if and only if  $v_1$, $v_2$, and $v_3$ are each extrema of the set of $x$-, $y$-, and $z$-values, respectively, of $K$'s vertices. 

\label{lem:bound}
\end{lemma} 

\begin{proof}
Let $[x_1,x_2]\times[y_1,y_2]\times[z_1,z_2]$ be the minimal bounding box of a lattice knot $K$. Every vertex in $K$ has an $x$-coordinate between $x_1$ and $x_2$, a $y$-coordinate between $y_1$ and $y_2$, and a $z$-coordinate between $z_1$ and $z_2$. 

Let ${\bf v}$ be a corner of the minimal bounding box. Then ${\bf v}$ is of the form ${\bf v}=(v_1,v_2,v_3)=(x_i,y_j,z_k)$, for $i,j,k\in\{1,2\}$. If $i=1,j=1$, or $k=1$, then the corresponding ${\bf v}$ coordinate is the minimum of all said coordinate values of vertices. If $i=2,j=2$, or $k=2$, then the corresponding ${\bf v}$ coordinate is the maximum of all said coordinate values of vertices. 

Now, let ${\bf v}=(v_1,v_2,v_3)\in V(K)$. If $v_1$ is less than or equal to all the $x$-values of vertices, then $v_1=x_1$ from the definition of the minimal bounding box. If $v_1$ is greater than or equal to all the $x$-values of vertices, then $v_1=x_2$. Apply the same argument to $v_2$ and $v_3$. Thus, $v$ is of the form ${\bf v}=(x_i,y_j,z_k)$, for $i,j,k\in\{1,2\}$ and is necessarily a corner of the minimal bounding box.

\end{proof}

\begin{lemma}
Let $K$ be a lattice knot.  Then $e_L(K)$ is even.  
\label{lem:even}
\end{lemma}
\begin{proof}
Since a knot is a closed curve, the sum of the $x^+$-stick lengths must equal the sum of the $x^-$-stick lengths, likewise for the $y$- and $z$-sticks. This implies that the sum of $x^+$- and $x^-$-stick lengths is even, similarly for the $y$- and $z$-sticks. The edge length of a knot is the sum of all stick lengths. Therefore, the edge length is the sum of three even positive integers. 
\end{proof}

\section{Conformations with $\delta_V=1$}
\label{section:Unknot}

\begin{theorem}
Let $K$ be a lattice conformation with $e_L(K)=\ell$.  Then $\delta_V(K)\leq \frac{\ell}{2}$.  
\end{theorem}

\begin{proof}
  \[ \delta_V(K)\leq \max\limits_{{\bf p,q}\in V(K)} d_{K}({\bf p,q} )\leq \frac{\ell}{2}.\]
\end{proof}

\begin{theorem}
For any knot $K$, $\delta_V([K])\geq 1$. For the unknot $U$, $\delta_V([U])=1$.  
\label{thm:distunknot}
\end{theorem}

\begin{proof}
For any lattice knot $K$ and any pair of points ${\bf p}$ and ${\bf q}$ in $V(K)$, $d_{K}({\bf p,q})\geq d_1({\bf p,q})$. Thus, $\delta_V(K)\geq 1$ implying that $\delta_V([K])\geq 1$. If $[K]=[U]$, consider the conformation $U=\partial [0,1]^2\times \{0\}$.  For any pair of points ${\bf p}$ and ${\bf q}$ in $V(U)$, $d_{K}({\bf p,q})=d_1({\bf p,q})$.  Therefore, $\delta_V(U)=1$ and $\delta_V([U])=1$.

\end{proof}

\begin{definition}
Two vertices ${\bf v}$ and ${\bf v'} \in V(K)$ are \emph{antipodal} if the two injective paths from ${\bf v}$ to ${\bf v'}$, along $K$, are of equal length.\end{definition}

We now prove Theorem \ref{thm:Unknot}. 

\begin{proof}[Proof of Theorem 1.1]

Let $K$ be a lattice knot with $\delta_V(K)=1$, and let ${\bf v}=(v_1,v_2,v_3) \in V(K)$. By Lemma \ref{lem:even}, $e_L(K)=2\ell$ for some positive integer $\ell$. Therefore, there exists a point ${\bf v'}\in V(K)$ antipodal to ${\bf v}$.  Let ${\bf r_1}$ and ${\bf r_2}$ be the two injective paths from ${\bf v}$ to ${\bf v'}$ along $K$. Note, \[1\leq \frac{d_K({ \bf v}, {\bf v}')}{d_1({ \bf v}, {\bf v'})}\leq \max\limits_{{\bf p,q}\in V(K)} \frac{d_K({\bf p,q})}{d_1({\bf p,q})} = \delta_V(K)=1.\]

This implies that the shortest injective path from ${\bf v}$ to ${\bf v'}$ along $K$ has a length equal to $d_1({\bf v},{\bf v'})$. Since the points ${\bf v}$ and ${\bf v'}$ are antipodal, both injectives paths have a length of $d_1(\bf{v,v'})$. Therefore, ${\bf r_1}$ and ${\bf r_2}$ are both staircase walks, by Lemma \ref{lem:staircasedistance}. 
 
Thus, if a coordinate function of ${\bf r_1}$ is nondecreasing, then the same coordinate function of ${\bf r_2}$ is nondecreasing, likewise if the coordinate function were nonincreasing.  Assume that the $x$-coordinate function of ${\bf r_1}$ is nondecreasing. Then the $x$-value of ${\bf r_1}(1)$ is greater than or equal to the $x$-value $v_1$. Since ${\bf r_1}(0)={\bf r_2}(0)$ and ${\bf r_1}(1)={\bf r_2}(1)$, the $x$-value of ${\bf r_2}(1)$ is greater than or equal  to the $x$-value  $v_1$. Thus, ${\bf r_2}$ has a nondecreasing $x$-coordinate function. An analogous argument applies for each component function and whether said function is nondecreasing or nonincreasing.

Assume that the $x$-coordinate functions of ${\bf r_1}$ and ${\bf r_2}$ are nondecreasing. Then all points in the image of ${\bf r_1}$ and ${\bf r_2}$ have $x$-coordinates greater than or equal to the $x$-value $v_1$. Since $K$ is the union of the image of ${\bf r_1}$ and ${\bf r_2}$, all vertices of $K$ have an $x$-coordinate greater than or equal to $v_1$. If the $x$-coordinate functions were nonincreasing, then all vertices of $K$ would have an $x$-coordinate less than or equal to $v_1$. A similar argument applies to the $y$- and $z$-coordinate functions of the paths. Thus, the $x$-, $y$-, and $z$-coordinates of all vertices are bounded by $v_1$, $v_2$, and $v_3$, respectively; therefore, we have that ${\bf v}$ is a corner of the minimal bounding box, Lemma \ref{lem:bound}.

Since ${\bf v}$ was an arbitrary vertex, each vertex is a corner of the minimal bounding box of $K$, and $K$ is contained in the boundary of the minimal boundary box. The boundary of the minimal boundary box is ambiently isotopic to $S^2$ and the only knot embeddable in such a surface is the unknot. 

\end{proof}

\begin{corollary}
The only lattice knot conformations of vertex distortion equalling one, up to isometry, are shown in Figure \ref{fig:distortion1}.
\end{corollary}

\begin{figure}[]

\includegraphics[scale=1.2]{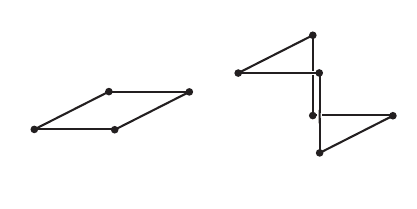}

\caption{Lattice conformations with vertex distortion equalling one.}
\label{fig:distortion1}
\end{figure}

\begin{proof} 
Since each vertex is a corner of the minimal bounding box, the minimal bounding box will be isometric to $[0,1]^3$ or $[0,1]^2$. If the minimal bounding box were any larger, then not all vertices of the knot could be corners of the box. Let us tabulate all knots embeddable in $[0,1]^3\cap \lat$. 

Starting at $(0,0,0)$, we can add an edge between $(0,0,0)$ and $(1,0,0)$ and an edge between $(0,0,0)$ and $(0,1,0)$. This is general for the tabulation since all 2 choose 3 options of  pairs of edges stemming from $(0,0,0)$ are isometric. 

There are three cases: 

Case 1: $ {\bf e_1}=(1,0,0)$ connected to $(1,1,0)$ and ${\bf e_2} =(0,1,0)$ connected to $(1,1,0)$.

Case 2: $ {\bf e_1}=(1,0,0)$ connected to $(1,0,1)$ and ${\bf e_2} =(0,1,0)$ connected to $(0,1,1)$. There are two ways to join $(1,0,1)$ and $(0,1,1)$ along the top of the box. The first way joins them to the point $(1,1,1)$ giving the right diagram in Figure \ref{fig:distortion1}. The second joins them to $(0,0,1)$ giving distortion $>1$. 

Case 3: ${\bf e_1}=(1,0,0)$ connected to $(1,0,1)$ and ${\bf e_2}=(0,1,0)$ connected to $(1,1,0)$. Then $(1,1,0)$ must be connected to $(1,1,1)$. There are two ways to connect $(1,0,1)$ to $(1,1,1)$ along the top of the box. The first is the short way just connecting the two points. The other way is the long way around. In either case, the distortion $>1.$

%We can now connect ${\bf e_1}=(1,0,0)$ to $(1,0,1)$ or $(1,1,0)$ and connect ${\bf e_2}=(0,1,0)$ to $(0,1,1)$ or $(1,1,0)$. 
%
%If ${\bf e_1}$ and ${\bf e_2}$ are both connected to their latter options, then we get the first diagram in Figure \ref{fig:distortion1}. If ${\bf e_1}$ and ${\bf e_2}$ are both connected to their former options, then we can get the knot in the second diagram of Figure \ref{fig:distortion1} or the knot $T$ following the edges connecting $(0,0,0)$ to $(1,0,0)$ to $(1,0,1)$ to $(0,0,1)$ to $(0,1,1)$ to $(0,1,0)$ to $(0,0,0)$; this traced knot $T$ does not have distortion equalling one.  
%
%If one of ${\bf e_i}$ is connected to its former option and the other is connected to its latter option, then the construction can be completed by a knot isometric to $T$, with distortion greater than one, or a knot isometric to one following the edges connecting $(0,0,0)$ to $(1,0,0)$ to $(1,0,1)$ to $(0,0,1)$ to $(0,1,1)$ to $(1,1,1)$ to $(1,1,0)$ to $(0,1,0)$ to $(0,0,0)$, with distortion greater than one.
\end{proof}

\section{Lattice Torus Knots}
\label{section:Lattice Torus Knots}

In \cite{ADAMS_2012} the authors illustrated a triplet of lattice knots showing minimal lattice stick number. However, no stick length nor vertification of knot type was given. In this section, we will tabulate a family of knots with similar geometry to that of \cite{ADAMS_2012}. In section \ref{section:torusverification} we prove that these knots are $(p,p+1)$-torus knots.

 For positive integers $p>2$, the tabulation is given in Figure \ref{fig:knotsequence}. Once verified as $(p,p+1)$-torus knots, a subsequence of these knots will be used to prove Theorem \ref{thm:Unbounded}. The sequence of stick types has $6p$ terms and is periodic with period 6. There are $2p$ $x$-sticks, $2p$ $y$-sticks, and $2p$ $z$-sticks. 
 
 We will denote by $\mathcal{K}_p$ the lattice curve defined by the tabulation in Figure \ref{fig:knotsequence}.

\begin{figure}[]

  \begin{tabular}{ c | *{3}{c} l }

     & $x$ & $y$ & $z$ \\   \cline{1-4}
    1 & 2 & $p-1$ & $2p-1$\\ 
    2 & 3 & $p$ & $2p-2$       \\
        3 & 3 & $p-1$ & $2p-3$ \\	  

    4 & 4 & $p$ & $2p-4$ \\
    5 & 4 & $p-1$ & $2p-5$ \\
    6 & 5 & $p$ & $2p-6$ \\
    7 & 5 & $p-1$ & $2p-7$  &  \hspace{1.5cm} $z^+$, $x^+$, $y^+$, $z^{-}$, $x^{-}$, $y^{-}$, \\
  
    \vdots & \vdots  & \vdots & \vdots  &  \hspace{1.5cm} $z^+$, $x^+$, $y^+$, $z^{-}$, $x^{-}$, $y^{-}$, \dots \\

    $2p-6$ & $p-1$ & $p$ & $6$  &  \hspace{1.5cm} $z^+$, $x^+$, $y^+$, $z^{-}$, $x^{-}$, $y^{-}$ \\

    $2p-5$ & $p-1$ & $p-1$ & $5$ \\
    $2p-4$ & $p$ & $p$ & $4$ \\
    $2p-3$ & $p$ & $p-1$ & $3$ \\
    $2p-2$ & $p+1$ & $p$ & $2$ \\
    $2p-1$ & $p$ & $2p-1$ & $1$ \\
   $2p$ & 1 & $p$ & $p$ \\

  \end{tabular}
     \vspace{1cm}

\caption{Stick length sequence with pairing stick type sequence of $\mathcal{K}_p$.}

\label{fig:knotsequence}

\end{figure}

\begin{lemma}
The tabulation given in Figure \ref{fig:knotsequence} forms a closed curve.
\label{lem:cc}
\end{lemma}

\begin{proof}
 The sum of $z^+$-stick lengths must equal the sum of the  $z^-$-stick lengths, similarly for the $x$- and $y$-sticks. The sequence of stick types dictates that the sum of the $z^+$-stick lengths, denoted $\sum |z^+|$, is the sum of the odd row entries in the $z$-column of Figure \ref{fig:knotsequence} while $\sum |z^-|$ is the sum of the even row entries. 

Then, 

 $$ \sum |z^+|=(2p-1)+(2p-3)+(2p-5)+\cdots+ 3+1=p^2,$$
and 

$$ \sum |z^-|=[(2p-2)+(2p-4)+(2p-6)+\cdots+4+2]+p=[p(p-1)]+p=p^2.$$

\noindent Likewise, 

$$\sum|y^+|=[p-1+p-1+\cdots+p-1]+2p-1=[(p-1)^2]+2p-1=p^2,$$

$$\sum|y^-|=[p+p+\cdots+p]+p=[p(p-1)]+p=p^2,$$

$$\sum|x^+|=[2+3+4+\cdots+p]+p=[(p+2)(p-1)/2]+p=p^2/2+3p/2-1,$$

\noindent and

$$\sum|x^-|=[3+4+5+\cdots+p+(p+1)]+1=[(p+4)(p-1)/2]+1=p^2/2+3p/2-1.$$\\

Therefore, these sticks form a closed curve.
\end{proof}

\begin{corollary}
The total length of the curve described by the tabulation in Figure \ref{fig:knotsequence} is \[5p^2+3p-2.\]
\label{cor:total}
\end{corollary}

%For a knot conformation $K$, stick in $K$ which is parallel to the $x$-axis is called an $x$-$stick$ of $K$, and $|K|_x$ denotes the number of $x$-sticks in $K$. Define $y$-sticks and $z$-sticks similarly. 

 Each $y$- and $z$-stick lies in a plane whose $x$-coordinate is some integer $a$; these sticks exist in the $x$-$level$ $a$ of the knot.
 
 \begin{definition} The \defn{x-level a}, for $a\in\Z$, of a lattice knot $K$ is the intersection of the plane $x=a$ and $K$. The \defn{y-level a} and \defn{z-level a} of $K$ is defined analogously for planes $y=a$ and $z=a$. \end{definition}

If each level of our closed lattice curve contains no points of self-intersection, then the curve is simple. The $n$th partial sum of the the $y$-stick length sequence will give the $y$-level containing the $n$th $y$-stick's terminal critical vertex.

 If the terms of the partial sum sequence of $y$-stick lengths are all distinct, then $\mathcal{K}_p$ contains just one arc in each $y$-level, similarly for the partial sum sequence of $x$- and $z$-stick lengths.

\begin{lemma} Self-intersections of $\mathcal{K}_p$  can only occur on levels that represent repeated values in the partial sum sequence of a given stick type's length sequence. \label{lem:int} \end{lemma}
\begin{proof}

First note that the stick type sequence  cycles $z$-, $x$-, $y$-sticks consecutively. 

The first $z$-stick terminates in $z$-level $2p-1$. The next two sticks, an $x$-stick and a $y$-stick, are contained in this $z$-level. On their own, these two sticks cannot form a point of self-intersection of the closed curve. The stick type sequence then returns to a $z$-stick that will necessarily change the $z$-level of the curve after its addition. 

If the curve never returns to $z$-level $2p-1$, meaning that no terminal critical vertex of another $z$-stick has a $z$-coordinate of $2p-1$, then no point of self-intersection is cause by an $x$- or $y$-stick intersecting either of the two sticks contained in $z$-level $2p-1$. This doesn't preclude  a $z$-stick intersecting the $x$- or $y$-stick in $z$-level $2p-1$; if a $z$-stick did intersect the $x$-stick in $z$-level $2p-1$, then the $y$-level that this intersection is contained in, must represent a repeated $y$-level traversed by the closed curve. This is a consequence of the previous argument. 

Upon entering a $y$-level, the closed curve will traverse a $z$- and $x$-stick. Among themselves, these two sticks will not produce a self-intersection point of the closed curve. Therefore, if a $z$-stick is to intersect an $x$-stick in this level, then the knot must return to this level later in its construction.

Combining the two previous paragraphs, if an $x$-stick in $z$-level $2p-1$ contains a point of self-intersection of the closed curve, this implies some $z$- or $y$-level was traversed multiple times by the closed curve, i.e. there exist repetition in the partial sum sequence of $z$- or $y$-stick lengths.

The argument generalizes for each level and each stick type. If a self-intersection were to occur, then the $\mathcal{K}_p$ construction must return to a repeated level. Returning to a repeated level corresponds to a repeated value in the partial sum sequence of a given stick type's length sequence.

\end{proof}

\begin{lemma} The only level with multiple arcs in $\mathcal{K}_p$ is $x$-level 2. \label{lem:x2} \end{lemma}

\begin{proof} The partial sum sequence of the $y$-stick length sequence is 
\[p-1,-1,p-2,-2,p-3,-3,\dots,1,1-p,p,0.\]

\noindent Ordering the values of this sequence in nondecreasing order, we obtain  $1-p,2-p,\dots,-2,-1,0,1,2,\dots,p-2,p-1,p$ and observe that  there are no repeated values in the sequence.

The partial sum sequence of the $z$-stick length sequence is 
$$2p-1,1,2p-2,2,2p-3,3,\dots,p-3,p+2,p-2,p+1,p-1,p,0.$$

\noindent Ordering the values of this sequence in nondecreasing order, we obtain $0,1,2,3,\dots,2p-1$ and observe that there are no repeated values in the sequence.

The partial sum sequence of the $x$-stick length sequence is 
$$2,-1,2,-2,2,-3,2,-4,\dots,2 ,2-p,2,1-p,1,0.$$

\noindent Excluding the $(2p-1)$st term, all odd indexed values are 2, and 2 is the only repeated value of this sequence. All $x$-levels excluding $x$-level 2 contain just one arc. 

The value 2 is repeated $p-1$ times in the partial sum sequence of the $x$-stick length sequence. Therefore, $x$-level 2 will have $p-1$ $y^+z^-$-stick arcs. \end{proof}

We will show that $x$-level 2 does not contain any self-intersection points, illustrating that no level of $\mathcal{K}_p$ contains a self-intersection point. 

\begin{lemma}
In the lattice knot $\mathcal{K}_p$, the $p-1$ $y^+z^-$-stick arcs in $x$-level 2 do not intersect. Moreover, these stick arcs lie in $x$-level 2 as shown in Figure \ref{fig:x2plane}.
\label{lem:closed}
\end{lemma}

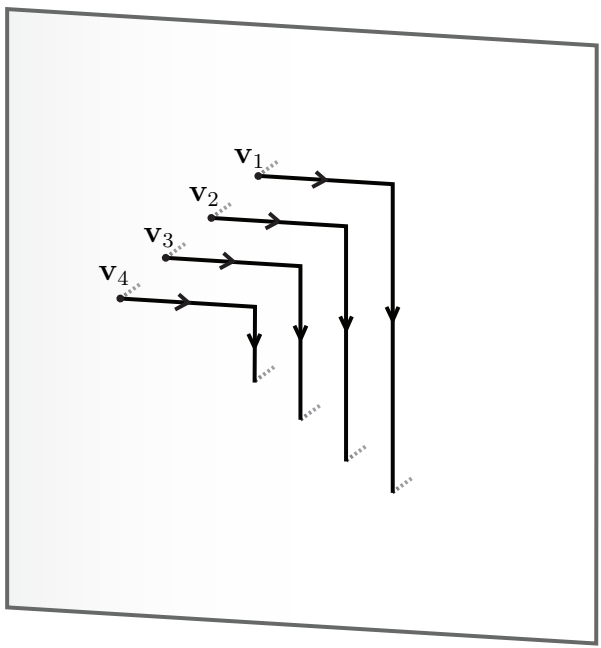
\begin{figure}[]

\begin{overpic}[unit=.5mm,scale=.5]{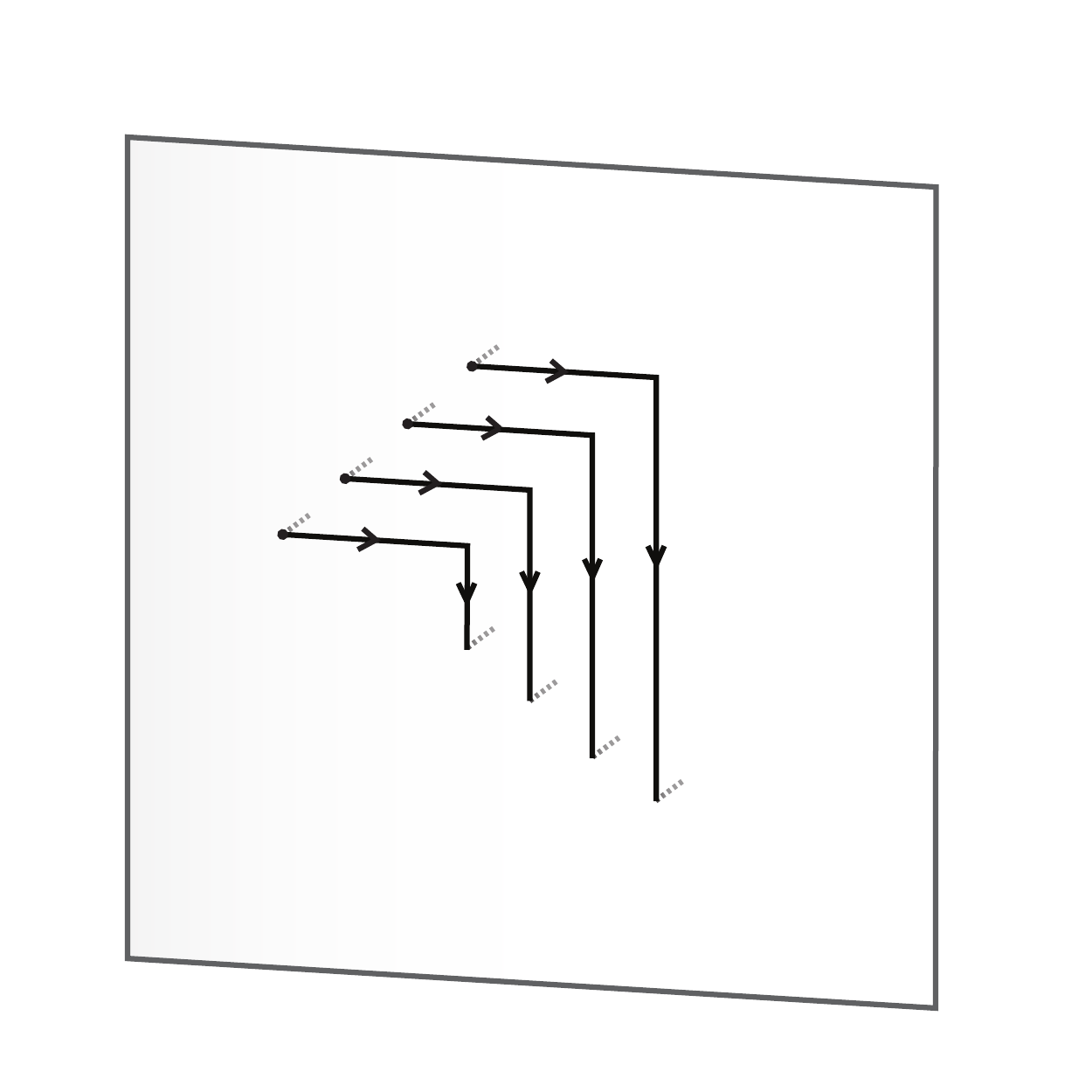} \put(85,144){\bf $\text{v}_1$}  \put(73,134){\bf $\text{v}_2$}  \put(61,123){\bf $\text{v}_3$}  \put(49,113){\bf $\text{v}_4$} 

\end{overpic}
\caption{The $x$-level 2 of $\mathcal{K}_5$.}
\label{fig:x2plane}
\end{figure}

%\begin{figure}[h]
%     \begin{overpic}[unit=.9mm,scale=.9]{../suspension.pdf} \put(136,46){$\bar a_1$} \put(72,46){$\bar a_2$}
%\put(50,46){$\bar a_3$} \put(80,79){$l_1$} \put(80,25){$l'_1$} \put(55,69){$l_2$} \put(55,34){$l'_2$} \put(44,63){$l_3$} \put(44,40){$l'_3$} \put(10,103){(0,1)} \put(10,0){(0,-1)}
%\end{overpic}
%\caption{The space $Z$.}
%\end{figure}

\begin{proof} The initial critical vertex of a $y^+z^-$-stick arc, in this plane, has a $y$ and $z$ value one less than the previous stick's initial critical vertex. Order the initial critical vertices of these arcs following the orientation of the closed curve; this gives the sequence ${\bf v}_1 ,{\bf v}_2, \dots, {\bf v}_{p-1}$ seen in Figure \ref{fig:x2plane}. 

Each column of stick lengths in Figure \ref{fig:knotsequence} can be used to define a sequence of vectors. Let ${\bf z}_i=(0,0,z_i)$ where $z_i$ is the value in the $i$th row of the $z$-column in the table of lengths; we define ${\bf x}_i=(x_i,0,0)$ and ${\bf y}_i=(0,y_i,0)$ likewise. Then, ${\bf v_1}={\bf z_1}+{\bf x_1}=(2,0,2p-1).$ We can then define a recursive sequence \[ {\bf v_n}={\bf v_{n-1}}+{\bf y_{2n-3}}-{\bf z_{2n-2}}-{\bf x_{2n-2}}-{\bf y_{2n-2}}+{\bf z_{2n-1}}+{\bf x_{2n-1}},\] for $2\leq n\leq p-1.$ We verify using the stick length sequences, \\
${\bf -x_{2n-2}+x_{2n-1}}=(0,0,0)$ for $2\leq n \leq p-1$, ${\bf y_{2n-3}-y_{2n-2}}=(0,-1,0)$ for $2\leq n \leq p-1$, and ${\bf -z_{2n-2}+z_{2n-1}}=(0,0,-1)$ for $2\leq n \leq p-1$. Therefore, we can simplify the former recursive definition to \[ {\bf v_{n}}={\bf v_{n-1}}+(0,-1,-1),\] for $2\leq n \leq p-1,$ and express the sequence in closed form as \[ {\bf v_{n}}=(2,1-n,2p-2-n),\] for $1\leq n \leq p-1.$ This verifies our claim that the initial critical vertex of a $y^+z^-$-stick arc, in $x$-level 2, has a $y$ and $z$ value one less than the previous stick's initial critical vertex. This implies that no two $y^+$-sticks on $x$-level 2 will intersect.

All $y^+$-sticks in $x$-level 2 have a length of $p-1$. 
%ccc Therefore, no terminal critical vertex of a $y^+$-stick in this plane lies above another $y^+$-stick, i.e. the $y$-value of any $y$-stick's terminal critical vertex is greater than all $y$-values of each sequential $y$-sticks in $x$-level 2. ccc 
Thus, no $z^-$-stick of a $y^+z^-$-arc will intersect a $y^+$-stick nor another $z^-$-stick on $x$-level 2, and, resultantly, $x$-level 2 contains no points of self-intersection. \end{proof}

\begin{corollary}
Figure \ref{fig:knotsequence} tabulates a lattice knot for $p>2$.
\end{corollary}

\begin{proof} Lemma \ref{lem:cc} gives that $\mathcal{K}_p$ is a closed curve. Together, lemma \ref{lem:int}, lemma \ref{lem:x2}, and lemma \ref{lem:closed} give that $\mathcal{K}_p$ is a simple curve.  \end{proof}

\noindent For $p=7$, the tabulation constructs the knot in Figure \ref{fig:t78}.

\begin{figure}

\includegraphics[scale=.28]{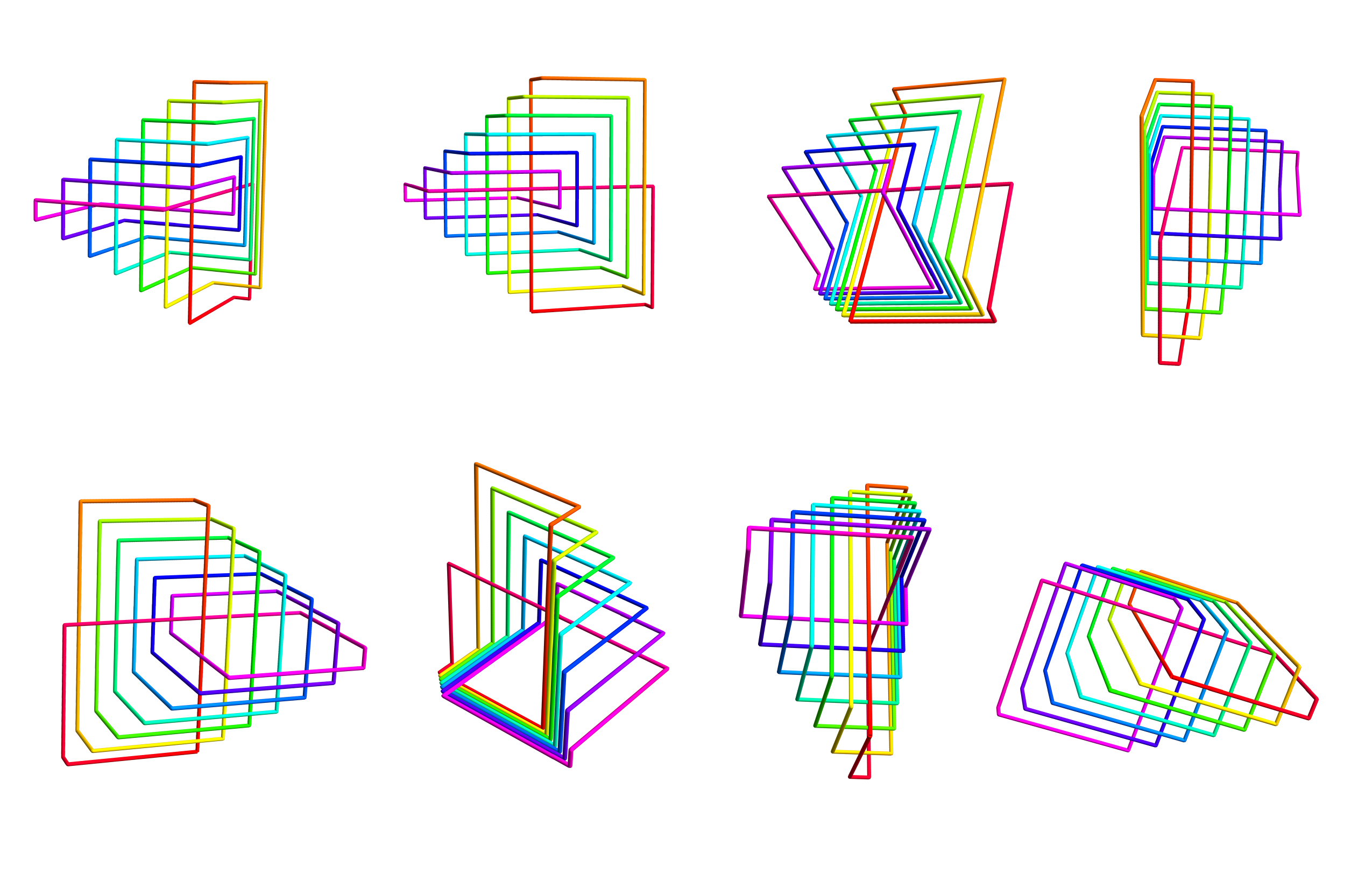}

\caption{Different views of our construction of $\mathcal{K}_p$.}
\label{fig:t78}
\end{figure}

%%%%%%%%%%%%%%%%%%%%%%%%%%%%%%%%%%%%%%%%%%%%%%%
%%%%%%%%%%%%%%%%%%%%%%%%%%%%%%%%%%%%%%%%%%%%%%%
%%%%%%%%%%%%%%%%%%%%%%%%%%%%%%%%%%%%%%%%%%%%%%%
%%%%%%%%%%%%%%%%%%%%%%%%%%%%%%%%%%%%%%%%%%%%%%%

\section{Torus Knot Verification}
\label{section:torusverification}
In order to verify that these knots are torus knots, we will generate a scalable toroidal polyhedron that $\mathcal{K}_p$ can be embedding into.

\begin{figure}

\includegraphics[scale=1]{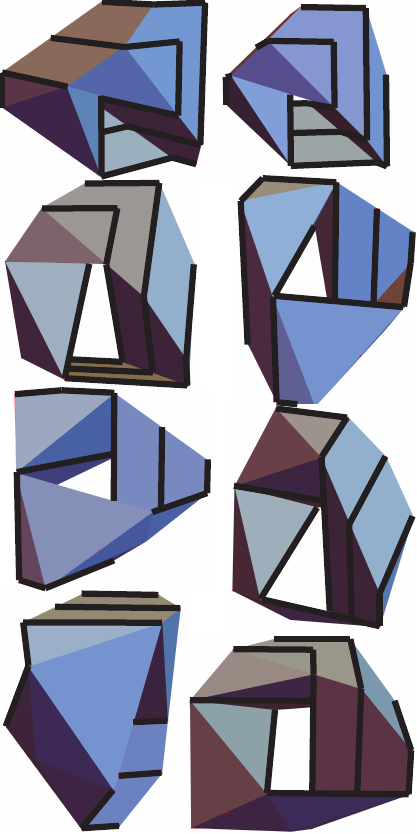}

\caption{Different views of a torus with embedded $\mathcal{K}_3$.}
\end{figure}

We will first show that, excluding the final $y^+$- and final $z^-$-stick, each stick type is coplanar. 

This was verified for the remaining $y^+$- and $z^-$-sticks in pursuit of proving our closed curve $\mathcal{K}_p$ has no points of self-intersection, Lemma \ref{lem:closed} and Figure \ref{fig:x2plane}. 
No terminal critical vertex of a $y^+$-stick in  $x$-level 2 lies above another $y^+$-stick, i.e. the $y$-value of any $y$-stick's terminal critical vertex is greater than all $y$-values of each sequential $y$-sticks in $x$-level 2. 

In fact, the terminal vertices of the $y^+$-sticks, which are also initial vertices of the $z^-$-sticks, are collinear. If the initial or terminal vertices of a collection of sticks of the same type is collinear, then the collection of sticks is coplanar.

We will use an analogous method to show the coplanarity of the remaining stick types; we will use the collinearity of critical vertices. The convex hull of each stick type will then be used as a face of the polyhedron.

As before, let ${\bf z}_i=(0,0,z_i)$ where $z_i$ is the value in the $i$th row of the $z$-column in the table of lengths and ${\bf x}_i=(x_i,0,0)$ and ${\bf y}_i=(0,y_i,0)$ are defined analogously.

\begin{lemma} Excluding the final $y^+$- and final $z^-$-stick, each stick type of $\mathcal{K}_p$,  is coplanar. \end{lemma}
\begin{proof}
$z^+$: We will show that all initial critical vertices of $z^+$-sticks are collinear. Let ${\bf w_n}$ represent the $n$th $z^+$-stick's initial critical vertex. Then, ${\bf w_1}=(0,0,0)$ and \[{\bf w_n}={\bf w_{n-1}}+{\bf z_{2n-3}}+{\bf x_{2n-3}}+{\bf y_{2n-3}}-{\bf z_{2n-2}}-{\bf x_{2n-2}}-{\bf y_{2n-2}},\] for $2\leq n \leq p$. We have  ${\bf z_{2n-3}}-{\bf z_{2n-2}}=(0,0,1)$, ${\bf y_{2n-3}}-{\bf y_{2n-2}}=(0,-1,0)$, and ${\bf x_{2n-3}}-{\bf x_{2n-2}}=(-1,0,0).$ Therefore, in closed form, \[{\bf w_n} = (1-n,1-n,n-1)=(1,1,-1)+(-1,-1,1)n ,\] for $1\leq n \leq p$.

%$x^+$: Likewise, ${\bf v_1}={\bf z_1}=(0,0,2p-1)$ and \[ {\bf v_n}= {\bf v_{n-1}}+{\bf x_{2n-3}}+{\bf y_{2n-3}}-{\bf z_{2n-2}}-{\bf x_{2n-2}}-{\bf y_{2n-2}}+{\bf z_{2n-1}},\] for $2\leq n \leq p$. .....

$y^-$: Each initial critical vertex of a $z^+$-stick is a terminal critical vertex of a $y^-$-stick. Therefore, ${\bf w_n}$ represents the terminal critical vertices of $y^-$-sticks. The collinearity of ${\bf w_n}$ implies that the $y^-$-sticks are coplanar. 

$x^+$: As seen in the earlier verification of $\mathcal{K}_p$ as a closed curve, Lemma \ref{lem:closed} associated with Figure \ref{fig:x2plane}, the initial critical vertices of the $y^+$-sticks, excluding the final $y^+$-stick, are collinear with  ${\bf v_n}=(2, 1, 2p-2)+ (0,-1,-1)n$ for $1\leq n\leq p-1$. Each initial critical vertex of a $y^+$-stick is a terminal critical vertex of a $x^+$ stick. Therefore, at least all but the final $x^+$-sticks are coplanar. 

The final $x^+$-stick has an initial critical vertex of $(1-p,1-p,p)$,  the penultimate $x^+$-stick has an initial critical vertex of $(2-p,2-p,1+p)$, and the $x^+$-stick previous to this has an initial critical vertex of $(3-p,3-p,2+p)$; these initial critical vertices are collinear. Therefore, the final three $x^+$-sticks are coplanar which implies that all the $x^+$-sticks are coplanar. 

$x^-$: Since all of the $y^-$-sticks' terminal critical vertices are collinear and all $y^-$-sticks are of equal length, the $y^-$-sticks have collinear initial critical vertices. Therefore, all $x^-$-sticks have collinear terminal critical vertices, and all $x^-$-sticks are coplanar. \end{proof}

\begin{theorem}
The lattice knot $\mathcal{K}_p$ from the tabulation given in Figure \ref{fig:knotsequence} is a $(p,p+1)$-torus knot.
\label{thm:torus}
\end{theorem}

\begin{proof}

Excluding the final $y^+$-stick and $z^-$-stick, the convex hull of a stick type intersects two other convex hull of stick types;  the convex hull of a stick type intersects the convex hull of the stick type prior and after in the stick type sequence. This produces a band with two half-twists, seen in Figure \ref{fig:band}, that $\mathcal{K}_p$ partially embeds into.

\begin{figure}[]

\includegraphics[scale=.4]{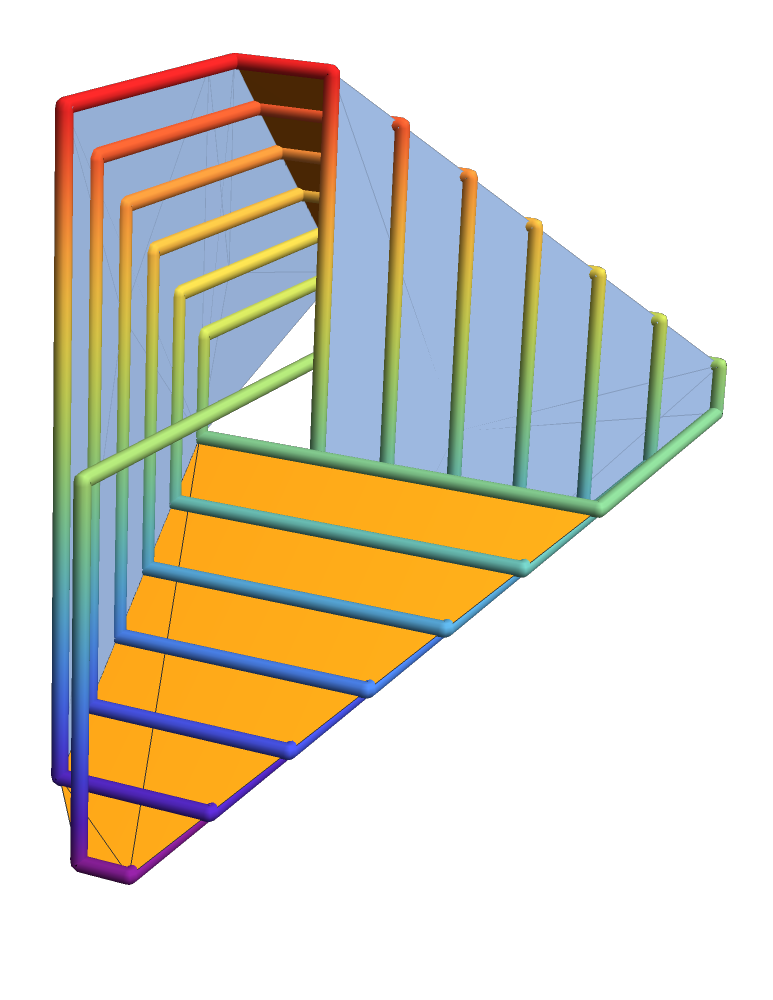}
\includegraphics[scale=.4]{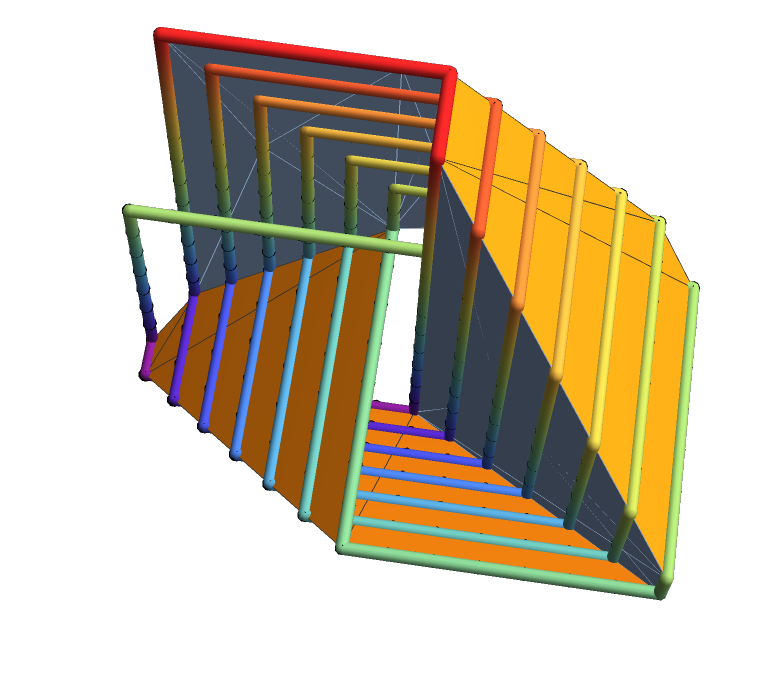}

\includegraphics[scale=.4]{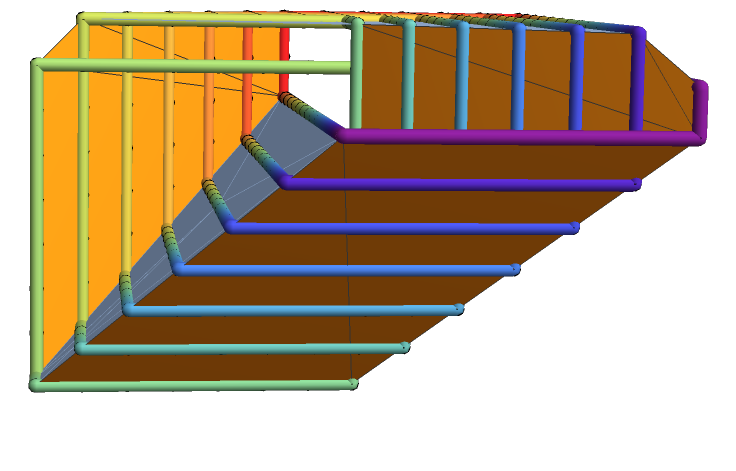}
\includegraphics[scale=.4]{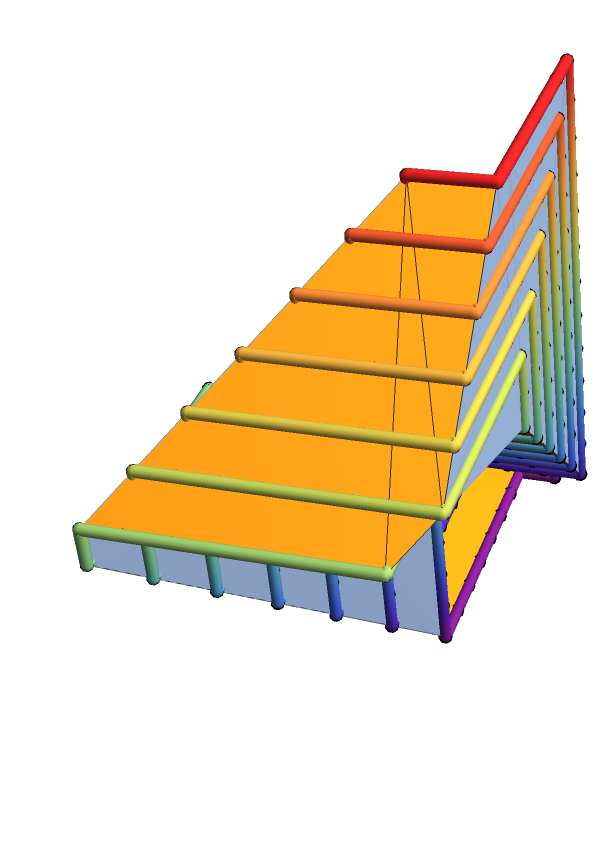}

\caption{ $\mathcal{K}_7$ superimposed with planar faces. }
\label{fig:band}
\end{figure}

Importantly, the geometry of this band remains fixed for all $\mathcal{K}_p$. Meaning, gluing of addition polygonal faces to Figure \ref{fig:band} to create the torodial polygon of Figure \ref{fig:embed}, is a general operation. The geometry of Figure \ref{fig:embed} welcomes an embedding of the general $\mathcal{K}_p$ into its boundary.  An explicit parameterization of the remaining faces is cumbersome and has been omitted for brevity.

\begin{figure}[]
\includegraphics[scale=.34]{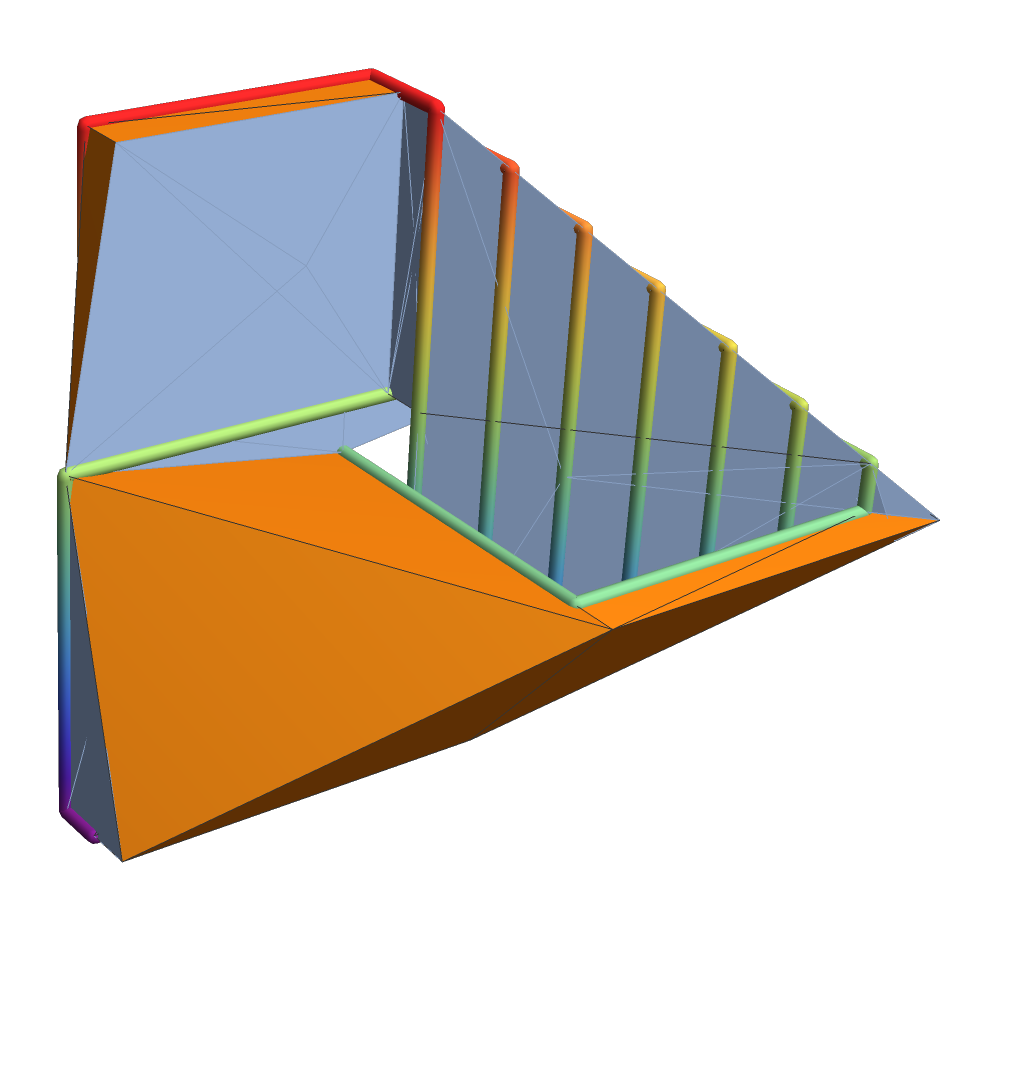}
\includegraphics[scale=.34]{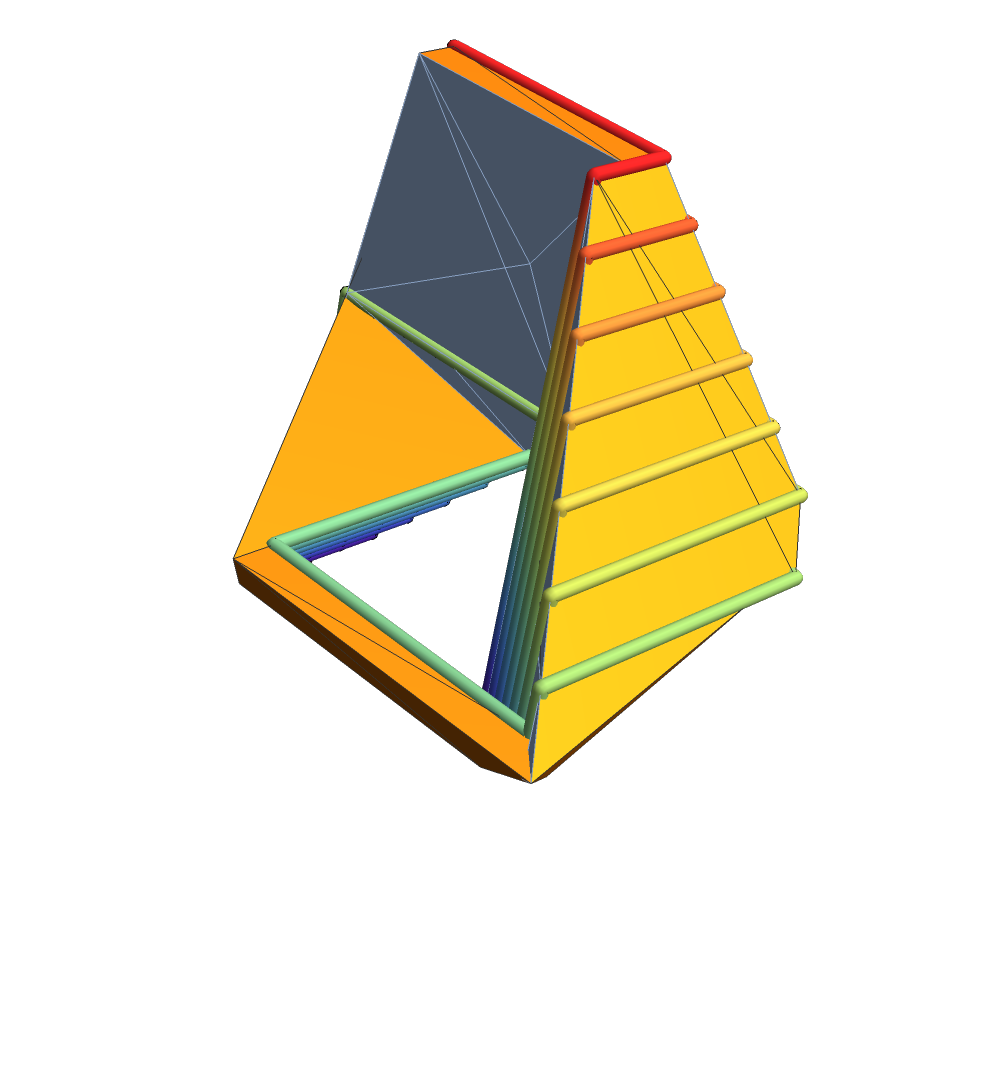}
\includegraphics[scale=.34]{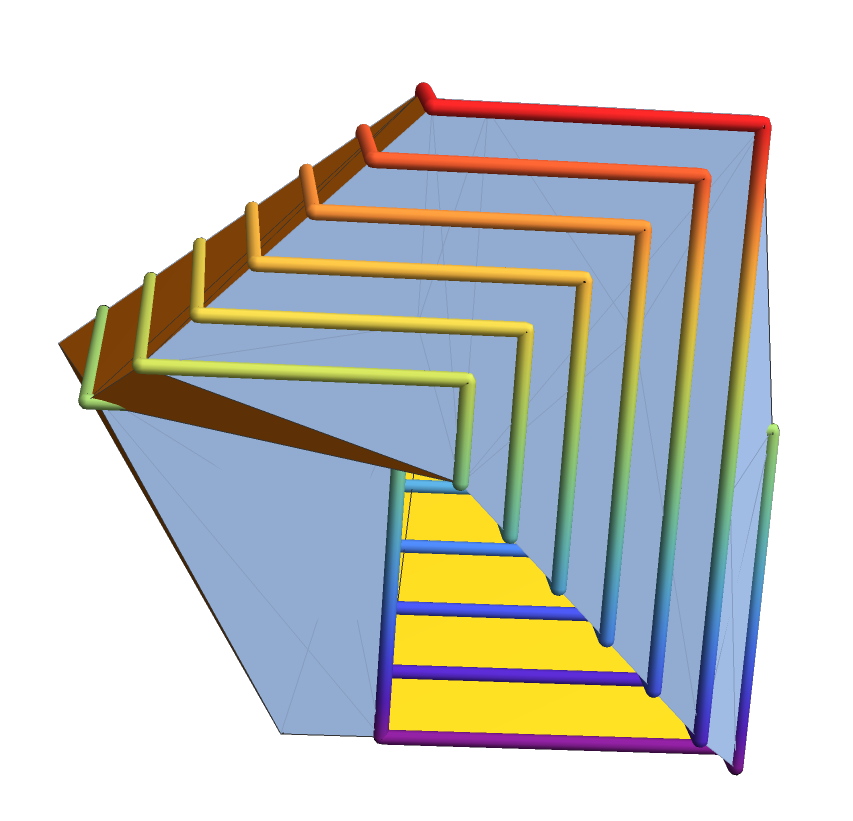}
\includegraphics[scale=.34]{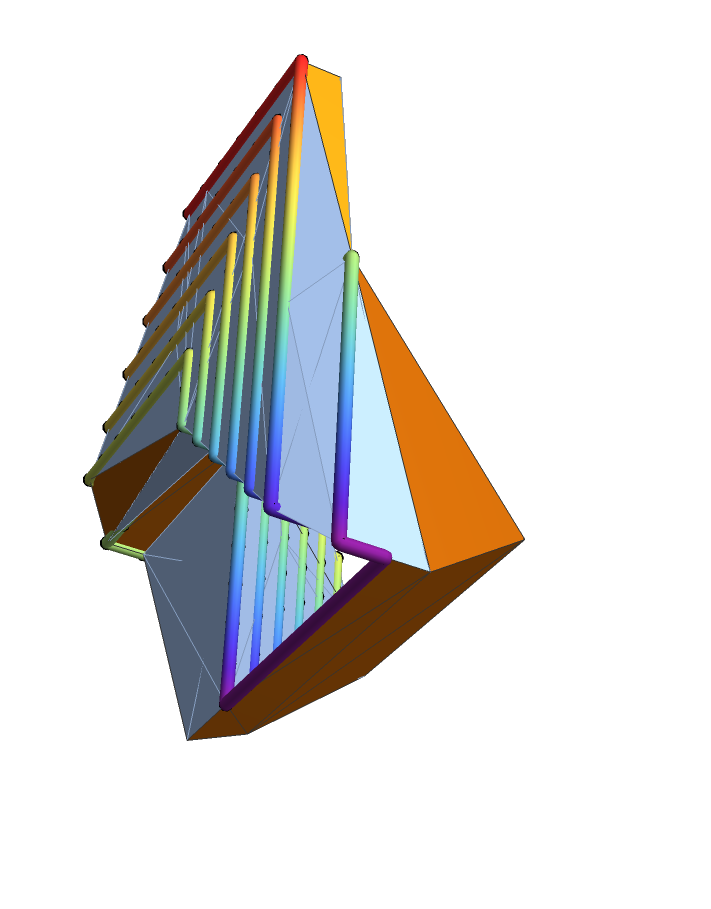}

\caption{Views of the constructed torus with embedded $\mathcal{K}_7$.}
\label{fig:embed}
\end{figure}

In determining the type of the torus knot, we draw the appropriate meridian on the constructed torus, Figure  \ref{fig:meridian}. The meridian meets the knot at each of the $p$ $y^-$-sticks. 

\begin{figure}[]

\includegraphics[scale=.9]{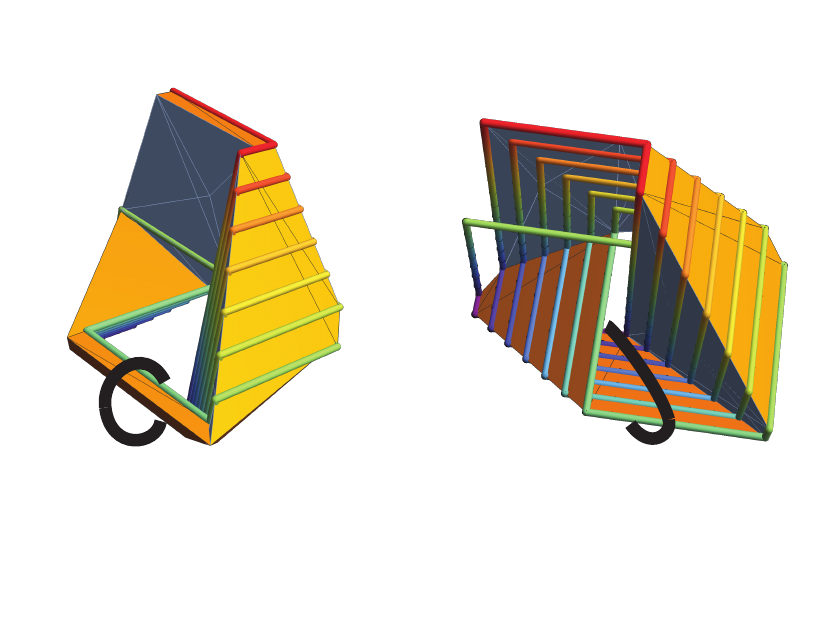}

\caption{Two views of a meridian drawn over the $\mathcal{K}_7$ knot embedded in the constructed polygonal torus and the knot with stick type faces.}
\label{fig:meridian}
\end{figure}

In Figure \ref{fig:longitude}, the longitude is drawn on the torus intersecting the $p$ $x^+$-sticks and the final $y^-$-stick.

\begin{figure}[]

\includegraphics[scale=1]{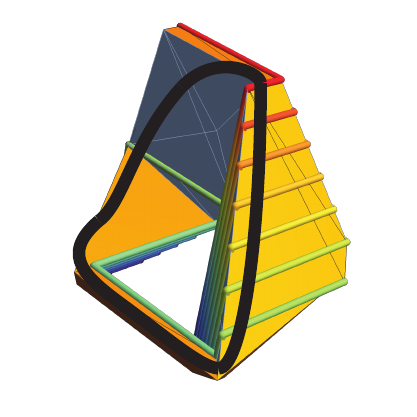}

\caption{Longitude drawn over the $\mathcal{K}_7$ knot embedded in the constructed polygonal torus.}
\label{fig:longitude}
\end{figure}

Altogether, $\mathcal{K}_p$ is a $(p,p+1)$ torus knot.

\end{proof}

\begin{theorem}
For $(p,p+1)$ torus knots, we have $s_{CL}([T_{p,p+1}])=6p$.
\end{theorem}

\begin{proof}

For any knot type,  $s_{CL}([K])\geq 6b[K]$, where $b[K]$ is the bridge index of $[K]$. The bridge index of a $(p,q)$ torus knot is $\min(p,q)$, so $s_{CL}([T_{p,p+1}])\geq 6p.$ Since we verified a lattice knot of $6p$ sticks to be a conformation of $T_{p,p+1}$, $s_{CL}([T_{p,p+1}])=6p.$

\end{proof}

\begin{corollary}
For each $p$, the lattice knot $\mathcal{K}_p$ realizes the minimal stick index $s_{CL}(\mathcal{K}_p)=s_{CL}([\mathcal{K}_p])=s_{CL}(T_{p.p+1}).$
\label{cor:min}
\end{corollary}

\section{Proof of Theorem \ref{thm:Unbounded}}
\label{section:main}

\begin{proposition}
For $p=2k\in 2\mathbb{Z}$, the vertex distortion of the lattice knot $\mathcal{K}_p$ from the tabulation given in Figure \ref{fig:knotsequence} satisfies \[ \delta_V(\mathcal{K}_p) \geq 9p^2/4+3p/2-1.\]
\label{prop:distort}
\end{proposition}

\begin{proof}
 Let $p=2k$. On the lattice knot $\mathcal{K}_p$, there is a point on the ${p/2}$ th $z^-$-stick that lies one away in space from a point on the final $y^+$-stick, see Figure \ref{fig:path}.

 \begin{figure}[]
\centering

\includegraphics[scale=.7]{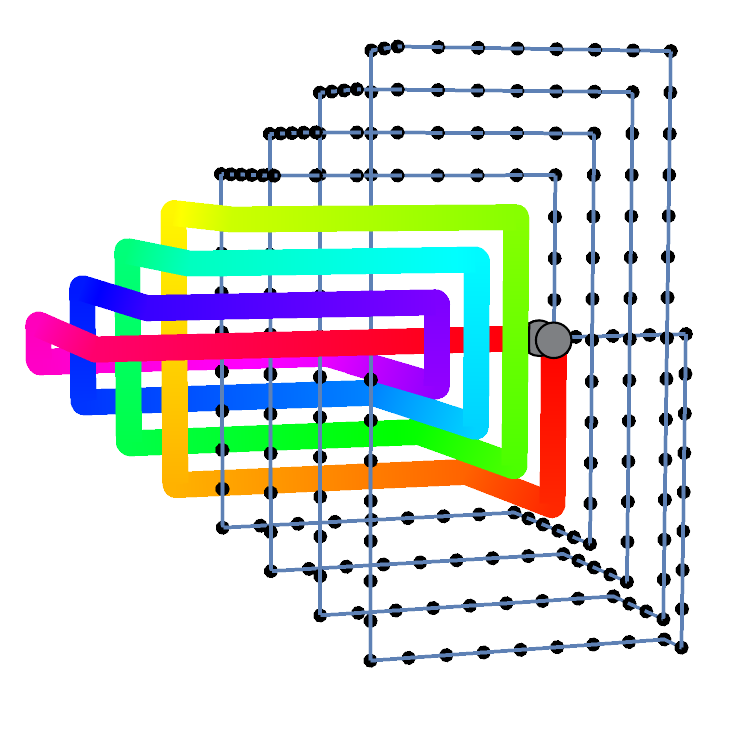}

\caption{The shortest path from the two points of $\mathcal{K}_8$ that realize its vertex distortion.}
\label{fig:path}
\end{figure}

 We will calculate the distortion value of this pair. Since the points lie one away in ambient space, the distortion value will be equal to the shorter of the two path lengths. Let us count the length of the path following the orientation, beginning at the point on the $z^-$-stick.

We begin with a partial $z^-$-stick of length $p/2$ and end with a $y^+$-stick of length $3p/2-1$. Of the remaining sticks, there are $p/2$, $y^-$-sticks; each of these sticks has a length of $p$. Thus, the total length of the $y^-$-sticks is \[p^2/2.\] 

There are $p/2-1$, $y^+$-sticks; each of these sticks has a length of $p-1$. Thus, the total length of the $y^+$-sticks is \[p^2/2-3p/2+1.\] 

The $x^+$-sticks have lengths following the sequence: $p/2+2, p/2+3, \dots, p/2+p/2, p$; this sequence has $p/2$ terms. The total length of the $x^+$-sticks is \[3 p^2/8 +3 p/4 -1.\]

The $x^-$-sticks have lengths following the sequence: $p/2+2, p/2+3, \dots, p/2+p/2, p/2+p/2+1$; this sequence has $p/2$ terms. The total length of the $x^-$-sticks is \[  3 p^2/8 + 3 p /4.\]

The length of the $z^-$ and $z^+$ sticks together follow the sequence: $p-1,p-2,\dots,3,2,1.$ The total length of the $z$-sticks is \[ p^2/2-p/2.\]

Therefore, the length of the path is \[ 9 p^2/4 +3p/2-1.\] Since the total length of the knot is $5p^2+3p-2$, Corollary \ref{cor:total} , the second path from one vertex to the other has a length of \[ 11 p^2/4+3p/2-1.\] This implies that the first path is the shorter of the two, and that \[9 p^2/4+3p/2-1\] is the distortion value of the two vertices. 

 This gives a lower bound for distortion on our $\mathcal{K}_p=\mathcal{K}_{2k}$ torus knots that increases quadratically as $p=2k$  increases. 

\end{proof}

\begin{proof}[Proof of Theorem \ref{thm:Unbounded}]
Theorem \ref{thm:torus} shows that $\mathcal{K}_{2k}$ is a $(2k,2k+1)$ torus knot. Corollary \ref{cor:min} shows that $\mathcal{K}_{2k}$ realizes the lattice-stick index of $[\mathcal{K}_{2k}]$, i.e. $\mathcal{K}_{2k}$ is a minimal lattice-stick number conformation. Proposition \ref{prop:distort} shows that $\mathcal{K}_{2k}$, as a sequence, satisfies \[\lim\limits_{k\to\infty} \delta_V(\mathcal{K}_{2k})=\infty.\]

%Recall that the vertex distortion of a lattice knot is defined as \[ \delta_V(K)=\max\limits_{{\bf a,b }\in V(K)} \frac{d_K({\bf a,b})}{d_1({\bf a,b})}.\]
%\noindent We will  find ${\bf a_{2k}} ,{\bf b_{2k}}\in V(T_{2k,2k+1})$ such that $ \lim\limits_{k\to \infty} \frac{d_K({\bf a_{2k},b_{2k}})}{d_1({\bf a_{2k},b_{2k}})}=\infty.$ This will prove Theorem \ref{thm:Unbounded} since \[ \frac{d_K({\bf a_{2k},b_{2k}})}{d_1({\bf a_{2k},b_{2k}})}\leq  \max\limits_{{\bf a,b}\in V(T_{2k,2k+1})} \frac{d_K({\bf a,b})}{d_1({\bf a,b})}=\delta_V(T_{2k,2k+1}).\] 
%
% We will analyze the distortion value of a pair of vertices in the aforementioned $T_{2k,2k+1}$ knots, see Figure \ref{fig:knotsequence} with $p=2k$. 

\end{proof}

% We have computed the distortion of all pairs for small $p$, in fact the distortion is equal to the lower bound for even $p$ less than 12; another formula, $11 p^2/4-p-11/4$, was verified through computation for odd $p$ less than 21. 
% 
%For even $p$ greater than 10, the distortion realizing points do not lie on the largest $y^+$-stick and $p/2$ th $z^-$-stick. For even $p$ greater than 12 and less than 24, we have verified that one distortion realizing point lies on the longest $y^+$-stick but the other lies on $p/2-1$ st $z^-$-stick.  
%
%There are only four sticks to add to the previous computation when counting the length of the one-larger-median pair of vertices. Subtracting this length from the total length of the knot gives the distortion lower bound of \[ 11 p^2/4-7p/2-5.\] This can be verified, computationally, to be the actual distortion for even $p$ greater than 10 and less than 24. 

%%%%%%%%%%%%%%%%%%%%%%%%%%%%%%%%%%%%%%%%%%%%%%%%%%%%%%%%%%%%%%%%%%%%%

\section{Conjectures}

The following conjectures naturally arise:  

\begin{conjecture}
$\lim\limits_{p,q\to\infty}\delta_V([T_{p,q}])\to \infty$
\end{conjecture} 

\begin{conjecture}
$\delta_V([K])=1 \iff [K]=[U]$
\label{conj:2}
\end{conjecture}

\noindent Theorem \ref{thm:distunknot} gives the `` if " direction of Conjecture \ref{conj:2}.

\bibliographystyle{hplain}   
    \bibliography{VertexDistortion}

%\bibliographystyle{amsplain}
%\bibliography{surgery}

\end{document}